\documentclass{amsart}
\usepackage{amsmath}
\usepackage{amsthm}
\usepackage{amssymb}
\usepackage{makecell} 
\usepackage{appendix}
\usepackage[german,english]{babel}
\usepackage[utf8]{inputenc}
\usepackage[new]{old-arrows}
\usepackage{epsfig,fancybox,color}
\usepackage{setspace}
\usepackage[linesnumbered,ruled]{algorithm2e}
\usepackage{tikz}
\usepackage{tikz-cd}
\usetikzlibrary{cd}
\usepackage[toc,page]{}
\usepackage{adjustbox}
\usepackage{todonotes}
\usepackage{yhmath}
\usepackage{mathtools}
\usepackage{stmaryrd}
\usepackage{fancyvrb} \VerbatimFootnotes
\usepackage{hyperref}
\usepackage{wasysym} 
\definecolor{darkred}{RGB}{160,0,0}
\definecolor{darkblue}{RGB}{0,0,160}
\hypersetup{
  colorlinks=true,
  citecolor=darkgreen,
  citecolor=darkblue,
  filecolor=black,
  linkcolor=darkblue,
  urlcolor=darkblue
}

\usepackage{xr}
\usepackage[margin=1.25in]{geometry} 

\tikzcdset{scale cd/.style={every label/.append style={scale=#1},
    cells={nodes={scale=#1}}}}

\DeclareMathAlphabet\mathbfcal{OMS}{cmsy}{b}{n} 

\DeclareFontFamily{U}{mathx}{\hyphenchar\font45}
\DeclareFontShape{U}{mathx}{m}{n}{
      <5> <6> <7> <8> <9> <10>
      <10.95> <12> <14.4> <17.28> <20.74> <24.88>
      mathx10
      }{}
\DeclareSymbolFont{mathx}{U}{mathx}{m}{n}
\DeclareMathSymbol{\bigtimes}{1}{mathx}{"91}

\newcommand{\excise}[1]{}

\theoremstyle{plain}
\newtheorem{thm}{Theorem}[section]
\newtheorem{lem}[thm]{Lemma}

\newtheorem{cor}[thm]{Corollary}
\newtheorem{prop}[thm]{Proposition}

\newtheorem{introthm}{Theorem}[section]

\theoremstyle{definition}
\newtheorem{defn}[thm]{Definition}

\newtheorem{example}[thm]{Example}

\theoremstyle{remark}

\newtheorem{notation}[thm]{Notation}

\numberwithin{equation}{section}

\renewcommand{\>}{\rangle}
\newcommand{\<}{\langle}

\newcommand{\NN}{\mathbb{N}}
\newcommand{\QQ}{\mathbb{Q}}
\newcommand{\RR}{\mathbb{R}}

\newcommand{\A}{\mathbb{A}}

\DeclareMathOperator{\Ban}{\mathbf{Ban}} 

\DeclareMathOperator{\LH}{\mathbf{LH}} 

\DeclareMathOperator{\TT}{\mathbb{T}} 

\DeclareMathOperator{\D}{\mathbf{D}} 
\DeclareMathOperator{\Ho}{H} 

\DeclareMathOperator{\IndBan}{\mathbf{IndBan}} %

\DeclareMathOperator{\id}{id} 

\DeclareMathOperator{\shHom}{\underline{\mathcal{H}\kern -.5pt\textit{om}}} 
\DeclareMathOperator{\RshHom}{R\underline{\mathcal{H}\kern -.5pt\textit{om}}} 
\DeclareMathOperator{\shExt}{\underline{\mathcal{E}\kern -1,75pt\textit{xt}}} 

\DeclareMathOperator{\Sh}{Sh} 
\DeclareMathOperator{\Loc}{Loc} 

\DeclareMathOperator{\Mod}{\mathbf{Mod}} 

\DeclareMathOperator{\HS}{\mathcal{HS}} %

\DeclareMathOperator{\I}{\mathbf{I}} 

\DeclareMathOperator{\op}{op} 

\DeclareMathOperator{\Dcap}{\mathcal{\wideparen{D}}} 
\DeclareMathOperator{\Dcapbd}{\mathcal{\wideparen{D}}\kern 1,00pt^{b}} 

\DeclareMathOperator{\Spa}{Spa} 
\DeclareMathOperator{\dR}{dR} 
\DeclareMathOperator{\pdR}{pdR} 
\DeclareMathOperator{\proet}{\text{pro\'{e}t}} %
\DeclareMathOperator{\BB}{\mathbb{B}}
\DeclareMathOperator{\OB}{\cal{O}\kern -1,00pt\mathbb{B}} 
\DeclareMathOperator{\normalOB}{O\kern -1,00pt B} 
\DeclareMathOperator{\normalOA}{O\kern -1,00pt A} 
\DeclareMathOperator{\normalOtildeA}{O\kern -1,00pt {\tilde{A}}} 
\DeclareMathOperator{\OA}{\cal{O}\kern -1,00pt\mathbb{A}} 
\DeclareMathOperator{\OBcap}{\wideparen{\cal{O}\kern -1,00pt\mathbb{B}}} 
\DeclareMathOperator{\OBdR}{\cal{O}\kern -1,00pt\mathbb{B}_{\dR}} 
\DeclareMathOperator{\OXBdR}{\cal{O}_{X}\kern -1,00pt\mathbb{B}_{\dR}} 
\DeclareMathOperator{\OXxXBdR}{\cal{O}_{X\times X}\kern -1,00pt\mathbb{B}_{\dR}} 
\DeclareMathOperator{\OC}{\cal{O}\kern -0,90pt\mathbb{C}} %
\DeclareMathOperator{\CB}{\cal{C}\kern -1,00pt\mathbb{B}} 

\DeclareMathOperator{\Otheta}{\cal{O}\kern -1,00pt\theta} 
\DeclareMathOperator{\Ovartheta}{\cal{O}\kern -1,00pt\vartheta} 
\DeclareMathOperator{\Oiota}{\cal{O}\kern -1,00pt\iota} 

\DeclareMathOperator{\Der}{Der} 
\DeclareMathOperator{\TB}{\mathcal{T}\kern -1,00pt\mathbb{B}} 
\DeclareMathOperator{\OmegaB}{\Omega\kern -1,00pt\mathbb{B}} 

\DeclareMathOperator{\DcapB}{\Dcap\kern -1,00pt\BB}

\DeclareMathOperator{\RB}{R\kern -1,00pt\BB}
\DeclareMathOperator{\PB}{P\kern -1,00pt\BB}
\DeclareMathOperator{\rhoB}{\rho\kern -1,00pt\BB}

\DeclareMathOperator{\SB}{S\kern -1,50pt\BB} 

\DeclareMathOperator{\dRB}{\dR\kern -1,50pt\BB}
\DeclareMathOperator{\DDB}{\DD\kern -1,50pt\BB}

\DeclareMathOperator{\SolB}{\Sol\kern -0.5pt\BB} 
\DeclareMathOperator{\dRfunctorB}{\dRfunctor\kern -1.5pt\BB} 

\DeclareMathOperator{\Conn}{Conn} 
\DeclareMathOperator{\LocB}{\Loc\kern -1,50pt\BB} 
\DeclareMathOperator{\ConnOB}{\Conn\kern -1,50pt\OB} 

\DeclareMathOperator{\DB}{\mathcal{D}\kern -1,00pt\mathbb{B}} 
\DeclareMathOperator{\EB}{\mathcal{E}\kern -0.40pt\mathbb{B}} 

\DeclareMathOperator{\Sol}{\mathcal{S}\kern -1.0pt\textit{ol}} 
\DeclareMathOperator{\nSolB}{Sol\kern -0,5pt\mathbb{B}} 
\DeclareMathOperator{\Rec}{\mathcal{R}\kern -1.0pt\textit{ec}} 

\DeclareMathOperator{\R}{R} 

\newcommand{\heart}{\ensuremath\heartsuit}

\DeclareMathOperator{\DD}{\mathbb{D}} 

\newcommand{\cal}[1]{\mathcal{#1}}

\DeclareMathOperator{\rL}{L} 

\DeclareMathOperator{\isomap}{\stackrel{\cong}{\longrightarrow}} 

\DeclareMathOperator{\dRfunctor}{\textit{d}\kern +0.5pt\mathcal{R}} 

\DeclareMathOperator{\lh}{^{\small\heart}\kern -3.00pt}

\begin{document}

\selectlanguage{english}

\title{The $p$-adic Cauchy Theorem and Overconvergent Period Sheaves}

\author{Finn Wiersig}
\address{National University of Singapore}
\email{fwiersig@nus.edu.sg}
\date{\today}

\begin{abstract}
The classical $p$-adic Cauchy theorem asserts that formal solutions of ordinary $p$-adic differential equations are convergent. In this article we establish a geometric analogue of this result for arbitrary smooth rigid-analytic varieties. More precisely, we show that the horizontal sections functor defined using the overconvergent period sheaf $\OB_{\dR}^{\dag,+}$ agrees with Scholze's horizontal sections functor defined using $\OB_{\dR}^{+}$. Equivalently, every formal solution arising from Scholze's construction is already overconvergent. As an application, we identify Scholze's horizontal sections functor with the de Rham functor for $\Dcap$-modules on vector bundles with flat connection.
\end{abstract}

\maketitle


\tableofcontents


\section{Introduction}

\subsection{The $p$-adic Cauchy theorem}

The classical Cauchy theorem, and more generally the Cauchy--Kovalevskaya theorem asserts that complex analytic differential equations admit local analytic solutions at ordinary points. Equivalently, formal solutions are automatically convergent. We refer to~\cite[\S4.3]{HTTDmodules} for an overview of the classical theory.

The situation over non-Archimedean fields is more subtle. For example, the solution
$\exp(z)=\sum_{n\ge0} z^n / n!$
of the differential equation $f'=f$ converges only on a disc of finite radius. Nevertheless, the non-Archimedean Cauchy theorem, originating in work of Lutz~\cite{Lutz1937}, asserts that at ordinary points formal solutions remain convergent.\footnote{See~\cite[Appendix III, Remark~2]{DworkGerottoSullivanGfunctions} for historical remarks.}

The purpose of this article is to reinterpret and generalise this theorem in the framework of $p$-adic Hodge theory. Let $X$ be a smooth rigid-analytic variety over a complete discretely valued field $k$ of mixed characteristic $(0,p)$. Solving $p$-adic differential equations directly is problematic. If $\cal{E}$ is a vector bundle with flat connection on $X$, the sheaf of horizontal sections $\cal{E}^{\nabla=0}$ need neither be locally constant nor determine $\cal{E}$ up to isomorphism.

A fundamental theorem of Scholze~\cite[Theorem~7.6]{Sch13pAdicHodge}, resolves this issue. The horizontal-sections functor
\begin{equation*}
\HS\left(\cal{E}\right)
:=
\left(
\nu^{-1}\cal{E}
\otimes_{\nu^{-1}\cal{O}}
\OB_{\dR}^{+}
\right)^{\nabla=0}
\end{equation*}
defines a fully faithful embedding from vector bundles with flat connection on $X$ into the category of $\BB_{\dR}^{+}$-local systems.\footnote{More generally, Scholze considers filtered vector bundles with integrable connection satisfying Griffiths transversality. Throughout this article all filtrations are trivial.}

In recent work~\cite{WiersigPeriods,WiersigReconstruction} we introduced an overconvergent period sheaf
$\OB_{\dR}^{\dag,+}\subset\OB_{\dR}^{+}$.
Locally on $X_{\proet}$, the sheaf $\OB_{\dR}^{+}$ is described by formal power series, whereas $\OB_{\dR}^{\dag,+}$ is described by convergent power series. Since $\OB_{\dR}^{\dag,+}$ carries a natural connection, one obtains a second horizontal-sections functor
\begin{equation*}
\HS^{\dag}\left(\cal{E}\right)
:=
\left(
\nu^{-1}\cal{E}
\otimes_{\nu^{-1}\cal{O}}
\OB_{\dR}^{\dag,+}
\right)^{\nabla=0}.
\end{equation*}

The guiding principle of this article is that $\HS(\cal E)$ computes formal solutions of the differential equation defined by $\cal E$, whereas $\HS^{\dag}(\cal E)$ computes convergent solutions. Our main result identifies the two.

\begin{introthm}[Theorem~\ref{thm:Cauchy}]
\label{introthm:easycompatibility}
For every vector bundle with flat connection $\cal{E}$ on $X$, the natural morphism
\begin{equation*}
\HS^{\dag}\left(\cal{E}\right)
\otimes_{\BB_{\dR}^{\dag,+}}
\BB_{\dR}^{+}
\isomap
\HS\left(\cal{E}\right)
\end{equation*}
is an isomorphism of $\BB_{\dR}^{+}$-local systems.
\end{introthm}

Thus every formal solution arising from Scholze's construction is already convergent. In this sense, Theorem~\ref{introthm:easycompatibility} may be regarded as a geometric form of the classical $p$-adic Cauchy theorem.

\subsection{Application to $\Dcap$-modules}

In previous work~\cite{WiersigReconstruction}
we constructed solution and de Rham functors for $\Dcap$-modules on rigid-analytic varieties,
\begin{equation*}
\SolB_{\pdR}^{\dag}
\colon
\D\left(\Dcap\right)^{\op}
\hookrightarrow
\D\left(\BB_{\pdR}^{\dag}\right),
\qquad
\dRfunctorB_{\pdR}^{\dag}
\colon
\D\left(\Dcap\right)
\hookrightarrow
\D\left(\BB_{\pdR}^{\dag}\right),
\end{equation*}
and proved that they are fully faithful on the category of $\cal C$-complexes. This category contains vector bundles with flat connection as well as many $\Dcap$-modules arising in the representation theory of $p$-adic analytic groups.

Theorem~\ref{introthm:easycompatibility} identifies these constructions with Scholze's horizontal-sections functor. More precisely, after fixing a model $\dRfunctorB_{\dR}^{\dag,+}$ whose scalar extension to $\BB_{\pdR}^{\dag}$ recovers $\dRfunctorB_{\pdR}^{\dag}$, we obtain:

\begin{introthm}[Theorem~\ref{thm:compatibility-scholze-cauchy}]
\label{introthm:easycompatibility-Dcap}
For every vector bundle with flat connection $\cal E$ on $X$, there is a canonical isomorphism of $\BB_{\dR}^{+}$-local systems
\begin{equation*}
\Ho^{-\dim X}
\left(
\dRfunctorB_{\dR}^{\dag,+}
\left(
\cal E
\right)
\right)
\otimes_{\BB_{\dR}^{\dag,+}}
\BB_{\dR}^{+}
\isomap
\HS\left(\cal E\right).
\end{equation*}
\end{introthm}

The proof reduces to identifying
\begin{equation*}
\Ho^{-\dim X}
\left(
\dRfunctorB_{\dR}^{\dag,+}
\left(
\cal E
\right)
\right)
\cong
\HS^{\dag}\left(\cal E\right),
\end{equation*}
which is achieved via a Poincaré lemma for the overconvergent period sheaf.

\begin{introthm}[Poincaré lemma, Theorem~\ref{thm:pioncare-lem}]
\label{thm:poincarelemma-intro}
The de Rham complex of $\OB_{\dR}^{\dag,+}$ is strictly exact:
\begin{equation*}
0
\to
\BB_{\dR}^{\dag,+}
\to
\OB_{\dR}^{\dag,+}
\stackrel{\nabla_{\dR}^{\dag,+}}{\to}
\OB_{\dR}^{\dag,+}
\widehat{\otimes}_{\nu^{-1}\mathcal O}
\nu^{-1}\Omega^1
\to
\cdots
\to
\OB_{\dR}^{\dag,+}
\widehat{\otimes}_{\nu^{-1}\mathcal O}
\nu^{-1}\Omega^{\dim X}
\to
0.
\end{equation*}
\end{introthm}

Compare this with Scholze's $\BB_{\dR}^{+}$-linear Poincaré lemma~\cite[Corollary~6.13]{Sch13pAdicHodge}.

\subsection{Outline of the proof}

The key observation is that locally on $X_{\proet}$ the sheaf $\OB_{\dR}^{+}$ is described by formal power series, whereas $\OB_{\dR}^{\dag,+}$ is described by convergent power series. Consequently, the functors $\HS$ and $\HS^{\dag}$ may be interpreted as assigning formal and convergent solutions to a differential equation.
After translating the problem into this language, the proof follows classical arguments. The main additional ingredient is a higher-dimensional version of the convergence estimates appearing in the classical one-variable theory.


\subsection{Ground fields and rings}\label{subsec:conventions-reconstructionpaper}
Throughout, $k$ denotes a complete discrete valuation field of mixed characteristic
$(0,p)$ with perfect residue field $\kappa$.
Fix a uniformiser $\pi\in k$ and write $k^{\circ}\subseteq k$
for the ring of power-bounded elements. Set $k_{0}:=W(\kappa)[1/p]$.
$C$ is the completion of a fixed algebraic closure of $\overline{k}$ of $k$.


\subsection{Acknowledgements}

I thank Tomoyuki Abe and Konstantin Ardakov for helpful discussions.


\section{A $p$-adic Cauchy Theorem}
\label{sec:CauchyCauchy}

\subsection{Overconvergent period sheaves}

We recall the period sheaves introduced by Scholze~\cite{Sch13pAdicHodge}
and their overconvergent counterparts constructed in~\cite{WiersigPeriods}.

Let $X$ be a smooth rigid-analytic $k$-variety and let
$\nu \colon X_{\proet}\to X$
denote its pro-étale site. We refer to~\cite[\S 3]{Sch13pAdicHodge}
for the construction of $X_{\proet}$ and recall only that affinoid
perfectoid objects form a basis of the site.

\begin{example}
  The adic spaces
  \begin{equation*}
  \widetilde{\TT}_{j}^{d}:=\Spa\left(
  k\langle T_{1}^{\pm1/p^{j}},\dots,T_{d}^{\pm1/p^{j}}\rangle,
  k^{\circ}\langle T_{1}^{\pm1/p^{j}},\dots,T_{d}^{\pm1/p^{j}}\rangle
  \right)
  \end{equation*}
  define a pro-étale cover
  $\widetilde{\TT}^{d}:=
  \text{``}\varprojlim\text{''}_{j\in\NN}\TT_{j}^{d}
  \to
  \TT^{d}$.
  After base change to a perfectoid extension of $k$, this becomes an affinoid
  perfectoid object of $\TT^{d}_{\proet}$.
\end{example}

The period sheaves $\BB_{\dR}^{\dag,+}$ and $\OB_{\dR}^{\dag,+}$ are sheaves
of ind-Banach spaces on $X_{\proet}$. We work systematically in this setting,
which provides a convenient framework for studying their cohomology and is
essential for the $\Dcap$-module constructions considered later.

For the purposes of the present section, we only require two basic properties.
First, there is a canonical morphism of sheaves of rings
\begin{equation*}
\nu^{-1}\cal O
\widehat{\otimes}_{k_{0}}
\BB_{\dR}^{\dag,+}
\to
\OB_{\dR}^{\dag,+},
\end{equation*}
which exhibits $\OB_{\dR}^{\dag,+}$ as a completion of
$\nu^{-1}\cal O\widehat{\otimes}_{k_{0}}\BB_{\dR}^{\dag,+}$.
Second, $\OB_{\dR}^{\dag,+}$ admits a particularly simple local description.
Suppose that $X$ is affinoid and equipped with an étale morphism
$X\to \TT^{d}$.
Let $\widetilde X:=X\times_{\TT^{d}}\widetilde{\TT}^{d}$
be the induced pro-étale cover. Writing $Z=\left(Z_{1},\dots,Z_{d}\right)$, there is the isomorphism
\begin{equation*}
\BB_{\dR}^{\dag,+}|_{\widetilde X}
\left<
Z/p^{\infty}
\right>
\isomap
\OB_{\dR}^{\dag,+}|_{\widetilde X},
\qquad
Z{i}\mapsto T_{i}-[T_{i}^{\flat}]
\end{equation*}
of $\BB_{\dR}^{\dag,+}|_{\widetilde X}$-ind-Banach algebras,
see~\cite[Theorem~\ref*{thm:localdescription-of-OBla}]{WiersigPeriods}.
Thus, locally on the pro-étale site, $\OB_{\dR}^{\dag,+}$ is a ring of
convergent power series over $\BB_{\dR}^{\dag,+}$. This description is the
starting point for our interpretation of $\OB_{\dR}^{\dag,+}$ as a sheaf of
convergent solutions to differential equations.

\subsection{Comparison with $\BB_{\dR}^{+}$ and $\OB_{\dR}^{+}$}
\label{subsubsec:OBdRplus}

We now compare the overconvergent period sheaves with Scholze's period sheaves
$\BB_{\dR}^{+}$ and $\OB_{\dR}^{+}$ from~\cite[\S 6]{Sch13pAdicHodge}.

Since $\BB_{\dR}^{\dag,+}$ and $\OB_{\dR}^{\dag,+}$ are sheaves of
ind-Banach algebras, whereas $\BB_{\dR}^{+}$ and $\OB_{\dR}^{+}$ are sheaves of
abstract rings, we shall frequently apply the forgetful functor $|\cdot|$
| which sends an ind-Banach object to its underlying abstract object,
see Appendix~\ref{sec:forgetIB} for details.

Let $U\in X_{\proet}$ be an affinoid perfectoid with $\widehat U=\Spa\left(R,R^{+}\right)$.
By~\cite[Theorem~6.5]{Sch13pAdicHodge}, the sections of $\BB_{\dR}^{+}$
over $U$ are given by the ring $\BB_{\dR}^{+}(R,R^{+})$.
Similarly,~\cite[Theorem~\ref*{thm:subsections-periodsheaves-affperfd}]{WiersigPeriods}
identifies the sections of $\BB_{\dR}^{\dag,+}$ with
$\BB_{\dR}^{\dag,+}(R,R^{+})$.
The construction of~\cite{WiersigPeriods} yields a canonical morphism
\begin{equation*}
|\BB_{\dR}^{\dag,+}(R,R^{+})|
\to
\BB_{\dR}^{+}(R,R^{+}),
\end{equation*}
and hence a canonical morphism of sheaves of rings
$|\BB_{\dR}^{\dag,+}|\to \BB_{\dR}^{+}$.

The comparison between $\OB_{\dR}^{\dag,+}$ and $\OB_{\dR}^{+}$ is slightly
more involved. 

\begin{notation}
  We denote the $p$-adically completed tensor product by
  $\widehat{\otimes}^{(p)}$.
\end{notation}

\begin{defn}[\cite{Sch13pAdicHodgeErratum}]
The sheaf $\OB_{\dR}^{+}$ is the sheafification of the following presheaf on $X_{\proet}$:
\begin{equation*}
U=\text{``}\varprojlim_{i\in I}\text{"}U_i
\mapsto
\varinjlim_{i\in I}
\varprojlim_{j\in\NN}
\left.
\left(
\mathcal O^{+}(U_i)
\widehat\otimes_{W(\kappa)}^{(p)}
\A_{\inf}(U)
\right)[1/p]
\middle/
\left(\ker\theta\right)^j
\right.
\end{equation*}
\end{defn}

\begin{lem}\label{lem:defn-OBdRplus-bd}
  Assume that $X$ is affinoid and equipped with an étale morphism
  $X\to\TT^{d}$. Let $U\in X_{\proet}$ be affinoid perfectoid.
  Fix a pro-étale presentation $U=\text{"}\varprojlim\text{"}_{i\in I}U_{i}$
  and $i\gg0$. Then
  \begin{equation*}
    \left(\BB_{\dR}^{+}(U)/\left(\ker\theta\right)^{j}\right)
    \llbracket Z \rrbracket/ (Z)^{j}
    \isomap
    \left.
    \left(\mathcal{O}^{+}(U_{i})\widehat{\otimes}_{W(\kappa)}^{(p)}\A_{\inf}(U)\right)[1/p]
    \middle/
    \left( \ker\theta\right)^{j}
    \right.,\quad
    Z_{l}\mapsto T_{l}-\left[T_{l}^{\flat}\right]
  \end{equation*}
  is an isomorphism of seminormed $k$-vector spaces for every $j\in\NN$.
\end{lem}

\begin{proof}
This is implicit in the proof of~\cite[Proposition~6.10]{Sch13pAdicHodge}.
The inverse map constructed there is bounded, hence the isomorphism is
compatible with the natural seminorms.
\end{proof}

Assume that $X$ is as in Lemma~\ref{lem:defn-OBdRplus-bd}. As in~\cite{WiersigPeriods},
we equip $\mathcal{O}^{+}\left(U_{i}\right)\widehat{\otimes}_{W(\kappa)}\A_{\inf}(U)$ with the
$\left(p,\xi\right)$-adic topology. There is a canonical morphism
\begin{equation}\label{eq:construct-OBdRdag-toOBdR}
  \mathcal{O}^{+}(U_{i})\widehat{\otimes}_{W(\kappa)}\A_{\inf}(U)
  \to
  \left.\left(\mathcal{O}^{+}(U_{i})\widehat{\otimes}_{W(\kappa)}^{(p)}\A_{\inf}(U)\right)[1/p]
  \middle/\left( \ker\Otheta_{\inf}\right)^{j}
  \right.,
\end{equation}
where the codomain carries the quotient seminorm.
Lemma~\ref{lem:defn-OBdRplus-bd} allows us to extend
\eqref{eq:construct-OBdRdag-toOBdR}
to a morphism
$|\OB_{\dR}^{\dag,+}| \to \OB_{\dR}^{+}$
of sheaves of $\nu^{-1}\mathcal{O}$-algebras.


\subsection{Formulation of the $p$-adic Cauchy Theorem}

The sheaf $\OB_{\dR}^{+}$ carries a canonical $\BB_{\dR}^{+}$-linear connection
\begin{equation*}
\nabla_{\dR}^{+}
\colon
\OB_{\dR}^{+}
\to
\OB_{\dR}^{+}
\otimes_{\nu^{-1}\cal O}
\nu^{-1}\Omega,
\end{equation*}
see~\cite[\S 6]{Sch13pAdicHodge}. Locally, this connection is given by formal differentiation.

\begin{lem}\label{lem:diff-OBdR+-cauchy}
  Let $X$ be affinoid and equipped with
  an étale map $X\to\TT^{d}$, giving rise to the
  pro-étale covering $\widetilde{X}:=X\times_{\TT^{d}}\widetilde{\TT}^{d}\to X$.
  Fix the following identification as in~\cite{Sch13pAdicHodgeErratum}:
  \begin{equation*}
    \BB_{\dR}^{+}|_{\widetilde X}
    \llbracket Z\rrbracket
    \cong
    \OB_{\dR}^{+}|_{\widetilde X},
    \qquad
    Z_l\mapsto T_l-[T_l^{\flat}],
  \end{equation*}
  Writing $dZ_{l}:=dT_{l}$, the connection $\nabla_{\dR}^{+}$ is given by
  \begin{equation*}
    \nabla_{\dR}^{+}\colon\OB_{\dR}^{+}\to\bigoplus_{i=1}^{d}\OB_{\dR}^{+}dZ_{i},\quad
      f\mapsto\sum_{i=1}^{d}\frac{d(f)}{dZ_{i}}dZ_{i}.
  \end{equation*}
\end{lem}

Lemma~\ref{lem:diff-OBdR+-cauchy} implies that
$\nabla_{\dR}^{+}$ restricts to a
$|\BB_{\dR}^{\dag,+}|$-linear connection
\begin{equation*}
|\nabla_{\dR}^{\dag,+}|
\colon
|\OB_{\dR}^{\dag,+}|
\to
|\OB_{\dR}^{\dag,+}|
\otimes_{\nu^{-1}\cal O}
\nu^{-1}\Omega.
\end{equation*}
We may therefore form horizontal sections with coefficients in either
$\OB_{\dR}^{+}$ or $|\OB_{\dR}^{\dag,+}|$. The main result of this section
is the following comparison theorem.

\begin{thm}[$p$-adic Cauchy theorem]
\label{thm:Cauchy}
Let $\cal E$ be a vector bundle with flat connection on $X$.
Then the canonical morphism
\begin{equation*}
\left(
\nu^{-1}\cal E
\otimes_{\nu^{-1}\cal O}
|\OB_{\dR}^{\dag,+}|
\right)^{\nabla=0}
\isomap
\left(
\nu^{-1}\cal E
\otimes_{\nu^{-1}\cal O}
\OB_{\dR}^{+}
\right)^{\nabla=0}
\end{equation*}
is an isomorphism of $\BB_{\dR}^{+}$-local systems.
\end{thm}


\subsection{Proof of Theorem~\ref{thm:Cauchy}}
\label{subsec:proofTheoremcompatible}

Throughout this subsection we suppress the ind-Banach structures and omit the
functor $|\cdot|$ from the notation.

The proof follows the classical argument~\cite[Theorem~9.6.1]{KedlaysBookDiffEq}.
Working locally on $X$, we choose a basis of the vector bundle with connection
$\cal{E}$. This induces bases in
$\nu^{-1}\cal E
\otimes_{\nu^{-1}\cal O}
\OB_{\dR}^{\dag,+}$
and
$\nu^{-1}\cal E
\otimes_{\nu^{-1}\cal O}
\OB_{\dR}^{+}$.
Using the Taylor series construction recalled Definition~\ref{defn:Taylorseries}),
we obtain an explicit basis of horizontal sections
$\left(
\nu^{-1}\cal E
\otimes_{\nu^{-1}\cal O}
\OB_{\dR}^{+}
\right)^{\nabla=0}$.
The main point is to show that these formal Taylor series are in fact
convergent. This follows from the convergence estimates established in
Proposition~\ref{prop:Mq-horizontals-sections-after-Mr}. Consequently, the
same basis already lies in
$\left(
\nu^{-1}\cal E
\otimes_{\nu^{-1}\cal O}
\OB_{\dR}^{\dag,+}
\right)^{\nabla=0}$.
It follows that both sides of Theorem~\ref{thm:Cauchy} are locally free and
that the comparison morphism identifies distinguished bases. The theorem is
therefore reduced to the convergence estimates for the associated Taylor
series.


\subsubsection{Convergent solutions}\label{subsubsec:convergingsolutions}

Fix a direct system $B^{0}\to B^{1}\to B^{2}\to\dots$ of commutative $k_{0}$-Banach algebras
such that, for all $q\in\NN$ and $f_{1},f_{2}\in B^{q}$, $\|f_{1}f_{2}\|\leq\|f_{1}\|\|f_{2}\|$.
Let $d\in\NN$, $Z=\left(Z_{1},\dots,Z_{d}\right)$, and $C^{q}:=B^{q}\left\<Z/p^{q}\right\>$
for all $q\in\NN$. We equip these with the bounded $B^{q}$-linear differentials
\begin{equation}\label{eq:differential-on-OB}
  d^{q}\colon
  C^{q}
  \cong
  B^{q}\widehat{\otimes}_{k_{0}}k_{0}\left\<Z/p^{q}\right\>
  \stackrel{\id\widehat{\otimes}d}{\longrightarrow}
  B^{q}\widehat{\otimes}_{k_{0}}\Omega_{k_{0}\left\<Z/p^{q}\right\>/k_{0}}
  =:\Omega_{C^{q}/B^{q}}.
\end{equation}
In the following, we fix a natural number $q\in\NN$.

\begin{defn}\label{defn:diff-Cq-module}
  A \emph{differential $C^{q}$-Banach module} $M^{q}$ is a
  finite free $C^{q}$-module equipped with a
  $B^{q}$-linear connection
  $\nabla^{q}\colon M^{q} \to \Omega_{C^{q}/B^{q}} \otimes_{C^{q}} M^{q}$.
\end{defn}

From now on, we fix a differential $C^{q}$-module $M^{q}$ with connection $\nabla^{q}$.

\begin{notation}
  For every $j=1,\dots,d$, $\nabla_{j}^{q}$ is the composition
  \begin{equation*}
    M^{q}
    \stackrel{\nabla^{q}}{\longrightarrow}
    \Omega_{C^{q}/B^{q}}\otimes_{C^{q}}M^{q}
    \cong\bigoplus_{l=1}^{d}M^{q}dZ_{l}
    \stackrel{\tau_{j}}{\longrightarrow}
    M^{q}dZ_{j}
    \cong M^{q}
  \end{equation*}
  where $\tau_{j}$ denotes the projection.
\end{notation}

\begin{defn}
  The \emph{$B^{q}$-module of horizontal sections} is $M^{q,\nabla=0}:=\ker\nabla^{q}$.
\end{defn}

\begin{lem}\label{lem:horizontalsections-via-ddZj}
  Let $m\in M^{q}$. Then $\nabla^{q}(m)=0$ if and only if $\nabla_{j}^{q}(m)=0$ for all $j=1,\dots,d$.
  In particular, $C^{q,\nabla=0}=B^{q}$ as $B^{q}$-Banach modules,
  where $C^{q}$ carries the connection $\nabla^{q}=d^{q}$.
\end{lem}

\begin{defn}\label{defn:Taylorseries}
  For any element $m\in M^{q}$,
  \begin{equation*}
    T\left(m\right):=\sum_{\alpha\in\NN^{d}}
    \frac{(-Z)^{\alpha}}{\alpha!}\nabla^{q,\alpha}\left(m\right)\in M^{q},
  \end{equation*}
  whenever the series converges.
  Here, $\nabla^{q,\alpha}:=\left(\nabla_{1}^{q}\right)^{\alpha_{1}}\circ\cdots\circ\left(\nabla_{d}^{q}\right)^{\alpha_{d}}$,
  $\alpha!:=\alpha_{1}!\cdots\alpha_{d}!$, and
  $|\alpha|:=\alpha_{1}+\cdots\alpha_{d}$
  for every $\alpha=\left(\alpha_{1},\dots,\alpha_{d}\right)\in\NN^{d}$.
\end{defn}

\begin{lem}\label{lem:T-additive-multiplicative}
  Consider a differential $C^{q}$-Banach module $M^{q}$. When the relevant series converge,
  $T\left(m+n\right)=T\left(m\right)+T\left(n\right)$ and
  $T(gm)=T(g)T(m)$ for all $g\in C^{q}$ and $m,n\in M^{q}$.
 \end{lem}

 \begin{lem}\label{lem:T-of-m-is-horizontal-section}
  Given a differential $C^{q}$-Banach module $M^{q}$, consider $m\in M^{q}$ such that
  $T(m)$ converges. Then $T(m)\in M^{q,\nabla=0}$.
  In particular, for $M^{q}=C^{q}$ with connection $\nabla^{q}=d^{q}$ and $g\in C^{q}$,
  $T(g)\in B^{q}$.
\end{lem}

\begin{proof}
  By Lemma~\ref{lem:horizontalsections-via-ddZj}, the first statement follows from
  \begin{align*}
    \frac{d}{dZ_{j}}\left(T\left(m\right)\right)
    =\sum_{\substack{\alpha\in\NN^{d} \\ \alpha_{i}\neq0}}
    \frac{\left(-Z\right)^{\alpha-e_{i}}}{\left(\alpha-e_{i}\right)!}\nabla^{q,\alpha}\left(m\right)
    +\sum_{\alpha\in\NN^{d}}\frac{(-Z)^{\alpha}}{\alpha!}\nabla^{q,\alpha+e_{i}}\left(m\right)
    =0
  \end{align*}
  for all $j=1,\dots,d$. Here, $e_{i}$ is the $i$th unit vector.
  The final sentence follows from Lemma~\ref{lem:horizontalsections-via-ddZj}.
\end{proof}

We would like to vary $q$, in the following sense.

\begin{lem}\label{lem:M-vary-to-Mr}
  For any natural number $r\geq q$,
  \begin{equation*}
    M^{q}\widehat{\otimes}_{C^{q}}C^{r}
    \rightarrow
    \Omega_{C^{q}/B^{q}}\widehat{\otimes}_{C^{q}}
    M^{q}\widehat{\otimes}_{C^{q}}C^{r}, \quad
    m\widehat{\otimes}g
    \mapsto
    \left(\nabla^{q}(m)\widehat{\otimes}g\right)
    +\left( m\widehat{\otimes}d^{r}(g)\right),
  \end{equation*}
  is a $B^{r}$-linear map which we denote
  by $\delta^{r}$. Then the following composition, denoted by $\nabla^{r}$,
  \begin{equation*}
    M^{r}
    :=M^{q}\widehat{\otimes}_{C^{q}}C^{r}
    \stackrel{\delta^{r}}{\rightarrow}
    \Omega_{C^{q}/B^{q}}\widehat{\otimes}_{C^{q}}
    M^{q}\widehat{\otimes}_{C^{q}}C^{r}
    \cong
    \Omega_{C^{q}/B^{q}}\widehat{\otimes}_{C^{q}}
    M^{r}
  \end{equation*}
  is a connection. That is,
  $\left(M^{r},\nabla^{r}\right)$ is a differential $C^{r}$-module.
\end{lem}

\begin{proof}
  One directly verifies that $\delta^{r}$ is well-defined and $\nabla^{r}$ satisfies the Leibniz rule.
\end{proof}

The following lemma is why we would like to vary $q$.

\begin{lem}\label{lem:T-converges-for-small-radius}
  Fix the notation from Lemma~\ref{lem:M-vary-to-Mr}.
  Denote the canonical morphisms $M^{q}\to M^{r}$ by $\iota^{q,r}$.
  Then there exists an $r_{0}\geq\max\{q,1\}$ such that
  for every $m\in M^{q}$ and $r\geq r_{0}$,
  $T\left(\iota^{q,r}(m)\right)\in M^{r}$ converges.
\end{lem}

\begin{proof}
  All the rings $C^{r}$ for $r\geq\max\{q,1\}$ carry norms.
  Therefore, every finite free $C^{r}$-module $M=\bigoplus C^{r}m_{i}$
  also carries the norm $\|\sum c_{i}m_{i} \|_{M}:=\max\|c_{i}\|_{C^{r}}$.
  We use this fact freely.

  Set $\varpi:=|p|^{\frac{1}{p-1}}\in\RR$
  and let $r_{0}\geq 1$ such that
  $|p|^{r_{0}}<\varpi\|\nabla^{q}\|^{-1}$.
  Fix an arbitrary $r\geq r_{0}$. We claim
  \begin{equation}\label{eq:genericradius-of-convergence-r}
    |p|^{r}<\varpi\|\nabla^{r}\|^{-1}.
  \end{equation}  
  To prove this, fix a $C^{q}$-basis $\left\{m_{1},\dots,m_{n}\right\}$ of $M^{q}$.
  Then $\left\{\iota^{q,r}\left(m_{1}\right),\dots,\iota^{q,r}\left(m_{n}\right)\right\}$
  is a $C^{r}$-basis of $M^{r}$. $M^{q}$ and $M^{r}$ carry the norms
  \begin{equation*}
    \|\sum_{j=1}^{n}g_{j}m_{j}\|_{M^{q}} = \max_{j=1,\dots,n}\|g_{j}\|_{C^{q}} \text{ and }
    \|\sum_{j=1}^{n}g_{j}\iota^{q,r}\left(m_{j}\right)\|_{M^{r}} = \max_{j=1,\dots,n}\|g_{j}\|_{C^{r}}.
  \end{equation*}
  Compute
  \begin{equation}\label{eq:bound-nablar-by-nabla-q}
  \begin{split}
    &\|\nabla^{r}\left(g_{j}\iota^{q,r}\left(m_{j}\right)\right)\|_{
       \Omega_{C^{q}/B^{q}}\widehat{\otimes}_{C^{q}}M^{r}} \\
    &=\|\delta^{r}\left(m_{j}\widehat{\otimes}g_{j}\right)\|_{
       \Omega_{C^{q}/B^{q}}\widehat{\otimes}_{C^{q}}
       M^{q}\widehat{\otimes}_{C^{q}}C^{r}} \\
    &=\|\left(\nabla^{q}\left(m_{j}\right)\widehat{\otimes}g_{j}\right)
      +\left( m_{j}\widehat{\otimes}d^{r}\left(g_{j}\right)\right)\|_{
       \Omega_{C^{q}/B^{q}}\widehat{\otimes}_{C^{q}}
       M^{q}\widehat{\otimes}_{C^{q}}C^{r}} \\
     &\leq\max\left\{
       \|\nabla^{q}\|\|m_{j}\|_{M^{q}}\|g_{j}\|_{C^{r}},
       \|m_{j}\|_{M^{q}}\|g_{j}\|_{C^{r}}
     \right\} \\
     &=\max\left\{
       \|\nabla^{q}\|\|g_{j}\|_{C^{r}},
       \|g_{j}\|_{C^{r}}
     \right\} \\
     &\leq\max\{\|\nabla^{q}\|,1\}\|g_{j}\|_{C^{r}} \\
     &=\max\{\|\nabla^{q}\|,1\}\|g_{j}\iota^{q,r}\left(m_{j}\right)\|_{M^{r}}
  \end{split}
  \end{equation}
  Hence,
  \begin{align*}
    \|\nabla^{r}\left(\sum_{j=1}^{n}g_{j}\iota^{q,r}\left(m_{j}\right)\right)\|_{
       \Omega_{C^{q}/B^{q}}\widehat{\otimes}_{C^{q}}M^{r}}
    &\leq\max_{j=1,\dots,n}\|\nabla^{r}\left(g_{j}\iota^{q,r}\left(m_{j}\right)\right)\|_{
       \Omega_{C^{q}/B^{q}}\widehat{\otimes}_{C^{q}}M^{r}} \\
    &\stackrel{\text{(\ref{eq:bound-nablar-by-nabla-q})}}{=}
    \max_{j=1,\dots,n}\max\{1 , \|\nabla^{q}\|\}\|g_{j}\iota^{q,r}\left(m_{j}\right)\|_{M^{r}} \\
    &=\max\{1 , \|\nabla^{q}\|\}\|\sum_{j=1}^{n}g_{j}\iota^{q,r}\left(m_{j}\right)\|_{M^{r}}.
  \end{align*}
  That is $\|\nabla^{r}\|\leq\max\{1,\|\nabla^{q}\|\}$.
  Now~(\ref{eq:genericradius-of-convergence-r}) follows from
  \begin{equation*}
    |p|^{r}\|\nabla^{r}\|
    \leq
    |p|^{r}\max\{1,\|\nabla^{q}\|\}
    \leq\max\{|p|^{r_{0}},|p|^{r_{0}}\|\nabla^{q}\|\}
    <
    \varpi.
  \end{equation*}
  In the last step, we have also used $r_{0}\geq1$.

  Finally, $T\left(\iota^{q,r}(m)\right)$ converges by the following computation:
   \begin{align*}
   \|\frac{(-Z)^{\alpha}}{\alpha!}\nabla^{r,\alpha}\left(\iota^{q,r}(m)\right)\|_{M^{r}}
   &\leq \|\frac{(-Z)^{\alpha}}{\alpha!}\|_{C^{r}} \|\nabla^{r,\alpha}\left(\iota^{q,r}(m)\right)\|_{M^{r}} \\
   &= \|\frac{p^{r|\alpha|}}{\alpha!}\frac{(-Z)^{\alpha}}{p^{r|\alpha|}}\|_{C^{r}} \|\nabla^{r,\alpha}\left(\iota^{q,r}(m)\right)\|_{M^{r}} \\
   &\leq |\frac{p^{r|\alpha|}}{\alpha!}|_{\QQ_{p}}\|\frac{(-Z)^{\alpha}}{p^{|\alpha|}}\|_{C^{r}} \|\nabla^{r}\|^{|\alpha|}\|\iota^{q,r}(m)\|_{M^{r}} \\
   &\leq \left(\frac{p^{-r}}{\varpi} \|\nabla^{r}\|\right)^{|\alpha|}\|\iota^{q,r}(m)\|_{M^{r}} \\
   &\stackrel{\text{(\ref{eq:genericradius-of-convergence-r})}}{\longrightarrow}0 \text{ for $|\alpha|\to\infty$.}    
  \end{align*}
  Here, $|-|_{\QQ_{p}}$ denotes the $p$-adic norm on $\QQ_{p}$.
  We also used $\varpi^{|\alpha|}\leq|\alpha!|$, cf.~\cite[Lemma 4.8]{KBB18}.
\end{proof}

\begin{prop}\label{prop:Mq-horizontals-sections-after-Mr}
  Fix a $C^{q}$-basis $\left\{ m_{1},\dots,m_{n}\right\}$ for $M^{q}$
  Fix the notation from Lemma~\ref{lem:T-converges-for-small-radius}.
  Fix $r_{0}$ such that for all $r\geq r_{0}$,
  the $T\left(\iota^{q,r}\left(m_{1}\right)\right),\dots,T\left(\iota^{q,r}\left(m_{n}\right)\right)\in M^{r}$
  converge. Then
  \begin{equation}\label{eq:Mq-horizontals-sections-after-Mr-the-map}
    \bigoplus_{j=1}^{n}B^{r}\cdot T\left(\iota^{q,r}\left(m_{j}\right)\right)
    \isomap M^{r,\nabla=0},
  \end{equation}
  as $B^{r}$-modules.
\end{prop}

\begin{proof}
  The $b_{j}:=T\left(\iota^{q,r}\left(m_{j}\right)\right)$ are horizontal sections,
  cf. Lemma~\ref{lem:T-of-m-is-horizontal-section}.
  Thus~(\ref{eq:Mq-horizontals-sections-after-Mr-the-map}) exists.
  Checking injcitivity and surjectivity is straightforward.
\end{proof}


\subsubsection{Formal solutions}
\label{subsec:formal-solutions-of-diff-eq}

Let $B$ be a commutative abstract
$k_{0}$-algebra, equipped with a morphism
$\varinjlim_{q\in\NN}B^{q}\to B$.
Set
$C:=B\left\llbracket Z \right\rrbracket$
and equip $C$ with the $B$-linear differential
\begin{equation*}
  d\colon C\to\bigoplus_{j=1}^{d}C_{j}dZ_{j}=:\Omega_{C/B},\quad
  f \mapsto \sum_{j=1}^{d}\frac{d}{dZ_{j}}\left(f\right)dZ_{j}.
\end{equation*}

\begin{defn}\label{defn:differential-Cfor-module}
  A \emph{differential $C$-module} $M$ is a
  $C$-module equipped with a $B$-linear connection
  $\nabla\colon M \to \Omega_{C/B} \widehat{\otimes}_{C} M$.
\end{defn}

Essentially all definitions from \S\ref{subsubsec:convergingsolutions}
go through verbatim for differential $C$-modules.

\begin{lem}\label{lem:M-vary-to-MdR}
  Let $M^{q}$ be a differential $C^{q}$-module with connection $\nabla^{q}$.
  For every $r\geq q$,
  \begin{equation*}
    M^{q}\otimes_{C^{q}}C
    \rightarrow
    \Omega_{C^{q}/B^{q}}\otimes_{C^{q}}
    M^{q}\otimes_{C^{q}}C,\quad
    m\otimes g
    \mapsto
    \left(\nabla^{q}(m)\widehat{\otimes}g\right)
    +\left( m\widehat{\otimes}d(g)\right),
  \end{equation*}  
  is a $B$-linear map which we denote by $\delta$.
  Then the following composition, denoted by $\nabla$,
  \begin{equation*}
    M
    :=M^{q}\otimes_{C^{q}}C
    \stackrel{\delta}{\rightarrow}
    \Omega_{C^{q}/B^{q}}\otimes_{C^{q}}
    M^{q}\otimes_{C^{q}}C
    \cong
    \Omega_{C/B}\otimes_{C}
    M
  \end{equation*}
  is a connection. We conclude that
  $\left(M,\nabla\right)$ is a differential $C$-module.
\end{lem}

\begin{prop}\label{prop:compatibility-with-Scholzesfunctor-plus-sections-formal}
  Fix the notation from
  the Lemma~\ref{lem:T-converges-for-small-radius} and~\ref{lem:M-vary-to-MdR}.
  Then the canonical morphism
  \begin{equation*}
    M^{r,\nabla=0}\otimes_{B^{r}}B\stackrel{\cong}{\longrightarrow}M^{\nabla=0}
  \end{equation*}
  is an isomorphism of $B$-modules.
\end{prop}

\begin{proof}
  This is a consequence of Proposition~\ref{prop:Mq-horizontals-sections-after-Mr}.
\end{proof}

\begin{cor}\label{cor:dager-formal-cauchy-formalism}
  Write $B^{\dag}:=\varinjlim_{r\to\infty}B^{r}$, $M^{\dag}=\varinjlim_{r\to\infty}M^{r}$. Then the canonical morphism
  \begin{equation*}
    M^{\dag,\nabla=0}\otimes_{B^{\dag}}B\stackrel{\cong}{\longrightarrow}M^{\nabla=0}
  \end{equation*}
  is an isomorphism of $B$-modules.
\end{cor}


\subsubsection{Proof of Theorem~\ref{thm:Cauchy}}\label{subsec:horizontalsectionsofsheavesindbanach}

We may assume that $X$ is affinoid and equipped with an étale morphism
$X\to\TT^{d}$. Set $\mathcal{M}^{\dag}:=\nu^{-1}\mathcal{E}\widehat{\otimes}_{\nu^{-1}\cal{O}}\OB_{\dR}^{\dag,+}$
and $\mathcal{M}:=\nu^{-1}\mathcal{E}\widehat{\otimes}_{\nu^{-1}\cal{O}}\OB_{\dR}^{+}$.
Fix an affinoid perfectoid $U\in X_{\proet}/\widetilde{X}$. Write
$B^{\dag}:=\BB_{\dR}^{\dag,+}(U)$,
$B:=\BB_{\dR}^{+}(U)$,
$M^{\dag}:=\mathcal{M}^{\dag}(U)$, and
$M:=\mathcal{M}(U)$.
Now apply Corollary~\ref{cor:dager-formal-cauchy-formalism}.
This completes the proof of Theorem~\ref{thm:Cauchy}. \hfill\qedsymbol


\section{Applications to $\Dcap$-modules}
\label{subsec:solutionderHamfunctorsDcapmod-Cauchy}

\subsection{The sheaf $\Dcap$}

Let $X$ be a smooth rigid-analytic $k$-variety. We write $\Dcap$ for the
sheaf of infinite-order differential operators on $X$ constructed by
Ardakov and Wadsley~\cite{AW19}. We use the notation and conventions of
\cite{WiersigPeriods}, where $\Dcap$ is regarded as a sheaf of
ind-Banach algebras.

The structure sheaf $\cal O$ is naturally a sheaf of left
$\Dcap$-modules. Hence
\begin{equation*}
\nu^{-1}\cal O
\widehat{\otimes}_{k_{0}}
\BB_{\dR}^{\dag,+}
\end{equation*}
is a $\nu^{-1}\Dcap$-$\BB_{\dR}^{\dag,+}$-bimodule object. In
\cite{WiersigPeriods} we proved that this bimodule structure extends
uniquely to a $\nu^{-1}\Dcap$-$\BB_{\dR}^{\dag,+}$-bimodule structure on
$\OB_{\dR}^{\dag,+}$.

We shall use the following refined form of this statement.
The period sheaves admit presentations
$\BB_{\dR}^{\dag,+}=\varinjlim_{q\geq0}\BB_{\dR}^{q,+}$ and
$\OB_{\dR}^{\dag,+}=\varinjlim_{q\geq0}\OB_{\dR}^{q,+}$
The sheaves $\OB_{\dR}^{q,+}$ have a simple local description.
Suppose that $X$ is affinoid and equipped with an étale morphism
$X\to \TT^{d}$. Let $\widetilde X:=X\times_{\TT^{d}}\widetilde{\TT}^{d}$
be the induced pro-étale cover. Writing $Z=\left(Z_{1},\dots,Z_{d}\right)$, there is the isomorphism
\begin{equation*}
\BB_{\dR}^{q,+}|_{\widetilde X}
\left<
Z/p^{q}
\right>
\isomap
\OB_{\dR}^{\dag,+}|_{\widetilde X},
\qquad
Z{i}\mapsto T_{i}-[T_{i}^{\flat}]
\end{equation*}
of $\BB_{\dR}^{q,+}|_{\widetilde X}$-ind-Banach algebras.
Using this local description, we are able to show that
$\OB_{\dR}^{q,+}$ carries a suitably unique
$\nu^{-1}\Dcap$-$\BB_{\dR}^{q,+}$-bimodule structure
for sufficiently large $q\gg0$. These structures are compatible with the transition maps in
$q$. Passing to the filtered colimit recovers the
$\nu^{-1}\Dcap$-$\BB_{\dR}^{\dag,+}$-bimodule structure on
$\OB_{\dR}^{\dag,+}$.


\subsection{Solution and de Rham functors}

We recall the solution and de Rham functors constructed
in~\cite[\S\ref*{subsec:solutionderHamfunctorsDcapmod-solutionpaper}]{WiersigReconstruction}.
We use the notation introduced in the preceding subsection.

Let $X_{\proet}\times\NN_{\gg0}^{\op}$
be the auxiliary site whose objects are pairs $(U,q)$, where
$U\in X_{\proet}$ and $q$ is sufficiently large. The lower bound on $q$
is chosen so that the bimodule structures on $\OB_{\dR}^{q,+}$ described
above are defined.

Let $X_{w}$ denote the category
whose objects are the affinoid subdomains of $X$ and whose morphisms
are the inclusions, carrying the weak Grothendieck topology.
There is a natural morphism of sites
$\lambda\colon
X_{\proet}\times\NN_{\gg0}^{\op}
\to
X_{w}$,
induced by the pro-étale projection to $X$. Concretely, for an affinoid
subdomain $U\subseteq X$, the inverse image of $U$ is represented by
$(\nu^{-1}(U),q)$ for $q$ sufficiently large.

On $X_{\proet}\times\NN_{\gg0}^{\op}$ we consider the sheaves
\begin{equation*}
\OB_{\dR}^{*,+}(U,q)
:=
\OB_{\dR}^{q,+}(U),
\qquad
\BB_{\dR}^{*,+}(U,q)
:=
\BB_{\dR}^{q,+}(U).
\end{equation*}
They are sheaves of $k_{0}$-ind-Banach algebras. There is a canonical
morphism
\begin{equation}\label{eq:OtensorbdRstarplus-to-OBdrstarplus-epi-cauchy}
\lambda^{-1}\mathcal O
\widehat{\otimes}_{k_{0}}
\BB_{\dR}^{*,+}
\to
\OB_{\dR}^{*,+}
\end{equation}
of sheaves of $k_{0}$-ind-Banach algebras. We equip
$\OB_{\dR}^{*,+}$ with the unique
$\lambda^{-1}\Dcap$-$\BB_{\dR}^{*,+}$-bimodule structure for which
\eqref{eq:OtensorbdRstarplus-to-OBdrstarplus-epi-cauchy} is a morphism
of $\lambda^{-1}\Dcap$-$\BB_{\dR}^{*,+}$-bimodule objects.

Let $\sigma\colon
X_{\proet}
\to
X_{\proet}\times\NN_{\gg0}^{\op}$
be the morphism of sites induced by passage to the filtered colimit over
$q$. Then there is a canonical morphism of sheaves of
$k_{0}$-ind-Banach algebras
$\sigma^{-1}\BB_{\dR}^{*,+}
\to
\BB_{\dR}^{\dag,+}$.
Thus $\sigma$ becomes a morphism of ringed sites
$\left(
X_{\proet},
\BB_{\dR}^{\dag,+}
\right)
\to
\left(
X_{\proet}\times\NN_{\gg0}^{\op},
\BB_{\dR}^{*,+}
\right)$.

For derived constructions we pass to the left heart
$\IndBan_{\I(k_{0})}
:=
\LH\left(\IndBan_{k_{0}}\right)$.
This is a Grothendieck abelian category~\cite[Lemma~4.20 and the remark
following Definition~3.2]{Bo21}; in particular, it has enough injectives.
Moreover, for any site $Y$, the category
$\Sh\left(Y,\IndBan_{\I(k_{0})}\right)$
carries a closed symmetric monoidal structure by~\cite[\S3.4]{Bo21}.
The functor
\begin{equation*}
\I\colon
\Sh\left(Y,\IndBan_{k_{0}}\right)
\to
\Sh\left(Y,\IndBan_{\I(k_{0})}\right)
\end{equation*}
is fully faithful and strongly monoidal, by~\cite[Corollary~1.2.28]{Sch99}
and~\cite[Lemma~3.33]{Bo21}. Since pullback commutes with $\I$,
cf.~\cite[Lemma~3.39(ii)]{Bo21}, the bimodule structure above induces a
$\lambda^{-1}\I\left(\Dcap\right)$-$\I\left(\BB_{\dR}^{*,+}\right)$-bimodule structure on
$\I\left(\OB_{\dR}^{*,+}\right)$.

Similarly, $\sigma$ gives rise to a morphism of ringed sites
\begin{equation*}
  \left(X_{\proet},\I\left(\BB_{\dR}^{\dag,+}\right)\right)
  \to \left(X_{\proet}\times\NN_{\gg0}^{\op},\I\left(\BB_{\dR}^{*,+}\right)\right).
\end{equation*}
Here, we have also used that pullback $\lambda^{-1}$ commutes with $\I$,
cf.~\cite[Lemma 3.39(ii)]{Bo21}. Now we can define the 
\emph{solution functor} via the formula
  \begin{equation*}
    \Sol:=\SolB_{\pdR}^{\dag}\colon
      \D\left(\I\left(\Dcap\right)\right)
      \to
      \D\left(\I\left(\BB_{\dR}^{\dag,+}\right)\right),
      \cal{M}^{\bullet}\mapsto
        \rL\sigma^{*}\R\shHom_{\lambda^{-1}\I\left(\Dcap\right)}
        \left(
          \lambda^{-1}\cal{M}^{\bullet},
          \I\left(\OB_{\dR}^{*,+}\right)
        \right).
  \end{equation*}
Let $d$ be the dimension of $X$.
The duality functor
$\DD\colon\D\left( \Dcap \right)
  \to\D\left( \Dcap \right)^{\op}$
has been introduced in~\cite[\S 7.2]{Bo21}.
Now we define the \emph{de Rham functor}:
\begin{equation*}
  \dRfunctor:=\dRfunctorB_{\dR}^{\dag,+}\colon
  \D\left( \I\left(\Dcap\right) \right) \to \D\left( \I\left(\BB_{\dR}^{\dag,+}\right)\right),
  \mathcal{M}^{\bullet} \mapsto \Sol\left(\DD\left(\mathcal{M}^{\bullet}\right)\right)\left[d\right].
\end{equation*}


\subsection{A Poincaré Lemma}
\label{subsec:connection-on-OBla-and-Pioncare-lem-cauchy}

We would like to describe $\dRfunctor\left(\cal{E}\right)$ for vector bundles with flat connection
$\cal{E}$. This can be done using the language of connections. As a first step,
we construct a $\BB_{\dR}^{\dag,+}$-linear connection on $\OB_{\dR}^{\dag,+}$,
coming from its $\nu^{-1}\Dcap$-$\BB_{\dR}^{\dag,+}$-bimodule structure
as in Theorem~\ref{thm:OBlaplus-bimodule}. This also allows us to formulate and prove
a Poincaré Lemma.

\begin{defn}\label{defn:tangentsheaf}
  By~\cite[\S 9.1, Proposition]{AW19},
  there is a coherent sheaf $\mathcal{T}$ on $X$ with $\mathcal{T}(U)=\Der_{k}\mathcal{O}(U)$
  for every affinoid subdomain $U\subseteq X$: the \emph{tangent sheaf}.
  Its sections are $k$-Banach spaces, so that we can it naturally as a sheaf of $k$-ind-Banach spaces.
\end{defn}

\begin{defn}\label{defn:nablaqlaplus}
  Let $X$ be affinoid. Fix an étale morphism $g\colon X\to\TT^{d}$
  and $q\ggg0$.
  We define the following map
  $\nabla_{\dR}^{q,+}\colon\OB_{\dR}^{q,+}\to\OB_{\dR}^{q,+}\widehat{\otimes}_{\nu^{-1}}\nu^{-1}\Omega^{1}$
  of sheaves of $\BB_{\dR}^{q,+}$-ind-Banach modules. The $\nu^{-1}\Dcap$-$\BB_{\dR}^{q,+}$-bimodule structure
  on $\OB_{\dR}^{q,+}$ from either Theorem~\ref{thm:OBlaplus-bimodule}
  gives a morphism
  \begin{equation*}
    \nu^{-1}\Dcap\to\shHom_{\BB_{\dR}^{q,+}}\left( \OB_{\dR}^{q,+},\OB_{\dR}^{q,+}\right)
  \end{equation*}
  of sheaves of $k_{0}$-ind-Banach algebras. The proof of \emph{loc. cit.}
  implies that this morphism is $\nu^{-1}\mathcal{O}$-linear. Thus it
  lifts to a morphism
  \begin{equation*}
    \nu^{-1}\Dcap\widehat{\otimes}_{\nu^{-1}\mathcal{O}}\OB_{\dR}^{q,+}
    \to\shHom_{\BB_{\dR}^{q,+}}\left( \OB_{\dR}^{q,+},\OB_{\dR}^{q,+}\right)
  \end{equation*}  
  of sheaves of $\OB_{\dR}^{q,+}$-ind-Banach algebras; here, the left-hand
  side becomes a monoid object through Lemma~\ref{lem:monoid-on-tensorproduct}.
  Compose it with the canonical map
  \begin{equation*}
    \nu^{-1}\mathcal{T}\widehat{\otimes}_{\nu^{-1}\mathcal{O}}\OB_{\dR}^{q,+}
    \to\nu^{-1}\Dcap\widehat{\otimes}_{\nu^{-1}\mathcal{O}}\OB_{\dR}^{q,+}.
  \end{equation*}
  to obtain
  \begin{equation*}
    \nu^{-1}\mathcal{T}\widehat{\otimes}_{\nu^{-1}\mathcal{O}}\OB_{\dR}^{q,+}
    \to\shHom_{\BB_{\dR}^{q,+}}\left( \OB_{\dR}^{q,+},\OB_{\dR}^{q,+}\right).
  \end{equation*}    
  Its dual appears in the following composition:
  \begin{equation*}
  \begin{split}
    \OB_{\dR}^{q,+}
    &\to\shHom_{\OB_{\dR}^{q,+}}\left(
    \shHom_{\BB_{\dR}^{q,+}}\left( \OB_{\dR}^{q,+},\OB_{\dR}^{q,+}\right),
    \OB_{\dR}^{q,+}\right) \\
    &\to\shHom_{\OB_{\dR}^{q,+}}\left( \nu^{-1}\mathcal{T}\widehat{\otimes}_{\nu^{-1}\mathcal{O}}\OB_{\dR}^{q,+},
    \OB_{\dR}^{q,+} \right) \\
    &\cong\shHom_{\nu^{-1}\mathcal{O}}\left( \nu^{-1}\mathcal{T},
    \OB_{\dR}^{q,+} \right) \\
    &\cong \OB_{\dR}^{q,+}\widehat{\otimes}_{\nu^{-1}\mathcal{O}}\nu^{-1}\Omega^{1}.
    \end{split}
  \end{equation*}
  This composition is $\nabla_{\dR}^{q,+}$.
\end{defn}

\begin{lem}
  Let $X$ be affinoid. Fix an étale morphism $g\colon X\to\TT^{d}$ and $q\gg0$.
  Then $\nabla_{\dR}^{q,+}$ fits into the commutative diagram
  \begin{equation*}
    \begin{tikzcd}
      \OB_{\dR}^{q,+} \arrow{r}{\nabla_{\dR}^{q,+}} &
      \OB_{\dR}^{q,+} \widehat{\otimes}_{\nu^{-1}\mathcal{O}}\nu^{-1}\Omega \\
      \nu^{-1}\mathcal{O} \arrow{u}\arrow{r}{\nu^{-1}\nabla} &
      \nu^{-1}\mathcal{O}\widehat{\otimes}_{\nu^{-1}\mathcal{O}}\nu^{-1}\Omega
      \arrow{u}
    \end{tikzcd}
  \end{equation*}
  of sheaves of $k$-ind-Banach spaces.
\end{lem}

\begin{proof}
  This is because the map $\nu^{-1}\mathcal{O}\to\OB_{\dR}^{q,+}$ is $\nu^{-1}\Dcap$-linear.
\end{proof}

In Lemma~\ref{lem:localdescription-nablalaplus} below,
we consider the usual differentials
$d^{q}\colon k_{0}\left\<Z/p^{q}\right\>\to\Omega_{k_{0}\left\<Z/p^{q}\right\>/k_{0}}$.

\begin{lem}\label{lem:localdescription-nablalaplus}
  Let $X$ be affinoid. Fix an étale morphism $g\colon X\to\TT^{d}$ and $q\gg0$.
  The étale map $g$ furnishes an isomorphism
  $\Omega^{1}\cong\bigoplus_{i=1}^{d}\mathcal{O}dT_{i}$.
  Together with~\cite[Corollary~\ref*{cor:localdescription-of-subsections-OBqplus}]{WiersigPeriods},
  this gives the vertical morphisms in the diagram
  \begin{equation}\label{cd:localdescription-nablalaplus-cauchy}
  \begin{tikzcd}
    \OB_{\dR}^{q,+}(U)
    \arrow{r}{\nabla_{\dR}^{q,+}(U)} &
    \bigoplus_{i=1}^{d}\OB_{\dR}^{q,+}(U)dT_{i} \\
    \BB_{\dR}^{q,+}(U) \widehat{\otimes}_{k_{0}}k_{0}\left\<Z/p^{q}\right\>
    \arrow{r}{\id\widehat{\otimes}d^{q}}
    \arrow{u}{\cong} &
    \BB_{\dR}^{q,+}(U) \widehat{\otimes}_{k_{0}}\Omega_{k_{0}\left\<Z/p^{q}\right\>/k_{0}}
    \arrow{u}{\cong}\arrow[swap]{u}{dZ_{i}\mapsto dT_{i}}
  \end{tikzcd}
  \end{equation}
  for every affinoid perfectoid $U\in X_{\proet}/\widetilde{X}$.
  This diagram~(\ref{cd:localdescription-nablalaplus-cauchy}) commute.
\end{lem}

\begin{proof}
  The vertical isomorphisms follows from~\cite[Corollary~\ref*{cor:localdescription-of-subsections-OBla}]{WiersigPeriods}
  Now we check the commutativity. 
  Consider
  \begin{equation*}
    \nabla_{\dR}^{q,+}(U)
    \colon
    \BB_{\dR}^{q,+}(U)\left\<Z/p^{q}\right\>
    \to
    \bigoplus_{i=1}^{d}
    \BB_{\dR}^{q,+}(U)\left\<Z/p^{q}\right\>dT_{i}.
  \end{equation*}
  This is a connection in the following sense: Firstly, it is, by definition,
  $\BB_{\dR}^{q,+}(U)$-linear. Secondly, it satisfies the Leibniz rule because the sections
  of $\mathcal{T}$ satisfy the Leibniz rule. Lastly, it vanishes on $\BB_{\dR}^{q,+}(U)$:
  By definition, $\nabla_{\dR}^{q,+}$ sends a section $b\in\BB_{\dR}^{q,+}(U)$
  to the morphism
  $\nu^{-1}\mathcal{T}\to\BB_{\dR}^{q,+}\left\<\frac{Z_{1},\dots,Z_{d}}{p^{q}}\right\>$,
  $P\mapsto bP(1)$ of sheaves of $\nu^{-1}\mathcal{O}$-ind-Banach modules
  on a localisation of $X_{\proet}$. But this map is zero, again
  because the sections of $\mathcal{T}$ are derivations, thus $P(1)=0$.
  
  Therefore, to show the commutativity of the diagram~(\ref{cd:localdescription-nablalaplus-cauchy}),
  it remains to compute $\nabla_{\dR}^{q,+}(U)\left(Z_{i}\right)$ for all $i=1,\dots,d$.
  This follows from Definition~\ref{defn:nablaqlaplus}, using that
  $\left\{dT_{i}\right\}_{i=1,\dots,d}$ is the $\mathcal{O}(U)$-basis
  of $\Omega^{1}(U)$ dual to $\left\{\partial_{i}\right\}_{i=1,\dots,d}$.
\end{proof}

Now we assume $X$ to be arbitrary. We can then proceed as in Definition~\ref{defn:nablaqlaplus}
to define a morphism
$\nabla_{\dR}^{\dag,+}\colon\OB_{\dR}^{\dag,+}\to\OB_{\dR}^{\dag,+}\widehat{\otimes}_{\nu^{-1}}\nu^{-1}\Omega^{1}$
of sheaves of $\BB_{\dR}^{\dag,+}$-ind-Banach modules.
We have $\nabla_{\dR}^{\dag,+}=\varinjlim\nabla_{\dR}^{q,+}$.
Thus Lemma~\ref{lem:localdescription-nablalaplus} implies that the underlying morphism
of abelian sheaves $|\nabla_{\dR}^{\dag,+}|$ recovers the connection considered
in \S\ref{sec:CauchyCauchy}.

\begin{thm}[Poincaré Lemma]\label{thm:pioncare-lem}
  The the de Rham complex is strictly exact:
  \begin{equation*}
    0
    \longrightarrow
    \BB_{\dR}^{\dag,+}
    \longrightarrow\OB_{\dR}^{\dag,+}
    \stackrel{\nabla_{\dR}^{\dag,+}}{\longrightarrow}\OB_{\dR}^{\dag,+}
    \widehat{\otimes}_{\nu^{-1}\mathcal{O}}
    \nu^{-1}\Omega^{1}
    \stackrel{\nabla_{\dR}^{\dag,+,1}}{\longrightarrow}\dots
    \stackrel{\nabla_{\dR}^{\dag,+,d-1}}{\longrightarrow}\OB_{\dR}^{\dag,+}
    \widehat{\otimes}_{\nu^{-1}\mathcal{O}}
    \nu^{-1}\Omega^{d}
    \to 0.
  \end{equation*}
\end{thm}

\begin{proof}
  We may assume that $X$ is affinoid
  and equipped with an étale morphism $X\to\TT^{d}$,
  giving rise to the pro-étale covering $\widetilde{X}\to X$.
  Then the exactness follows from~\ref{lem:localdescription-nablalaplus}.
  The strict exactness follows from a variant of the open mapping theorem, see~\cite[Theorem 4.9]{Ba15}.
\end{proof}


\subsection{Solutions of $\cal{O}$-modules with integrable connection}
\label{subsec:sol-of-MIC}

\begin{defn}
  A sheaf of $\cal{O}$-ind-Banach modules is \emph{locally finite free}
  if it is, locally, isomorphic to a finite direct sum of copies of $\cal{O}$
  as a sheaf of $\cal{O}$-ind-Banach modules.
\end{defn}

We would like to give a simpler description of the solution functor
$\Sol=\SolB_{\dR}^{\dag,+}$
on the category of $\Dcap$-ind-Banach modules which are locally finite free
as $\cal{O}$-modules. To do this, we fix an arbitrary complex $\cal{M}^{\bullet}$
of $\I\left(\Dcap\right)$-module objects on $X$. We continue to
fix the notation as in \S\ref{subsec:solutionderHamfunctorsDcapmod-solutionpaper}.
The canonical morphisms $\OB_{\dR}^{q,+}\to\OB_{\dR}^{\dag,+}$ induce
a morphism $\OB_{\dR}^{*,+}\to\sigma_{*}\OB_{\dR}^{\dag,+}$ of $\lambda^{-1}\Dcap$-$\BB_{\dR}^{*,+}$-bimodule
objects, cf. the proof of Lemma~\ref{lem:OBdRstarplus-bimodule}.
As $\I$ is strongly monoidal, cf.~\cite[Lemma 3.33]{Bo21}, we get the morphism
\begin{equation*}
   \I\left(\OB_{\dR}^{*,+}\right)\to\R\sigma_{*}\I\left(\OB_{\dR}^{\dag,+}\right)
\end{equation*}
in complexes of 
$\lambda^{-1}\I\left(\Dcap\right)$-$\I\left(\BB_{\dR}^{*,+}\right)$-bimodule objects.
Next, this gives
\begin{equation}\label{eq:solexplicitdescription-MIC-constructmap-anotherlemma}
   \R\shHom_{\lambda^{-1}\I\left(\Dcap\right)}\left(
     \lambda^{-1}\cal{M}^{\bullet},
     \I\left(\OB_{\dR}^{*,+}\right)
     \right)
   \to
   \R\shHom_{\lambda^{-1}\I\left(\Dcap\right)}\left(
     \lambda^{-1}\cal{M}^{\bullet},
     \R\sigma_{*}\I\left(\OB_{\dR}^{\dag,+}\right)
     \right),
\end{equation}
which is a morphism
of $\left(\BB_{\dR}^{*,+}\right)$-module objects. Now use
$\nu^{-1}=\left(\lambda\circ\sigma\right)^{-1}=\sigma^{-1}\circ\lambda^{-1}$ to compute
\begin{equation}\label{eq:solexplicitdescription-MIC-constructmap}
\begin{split}
   \R\shHom_{\lambda^{-1}\I\left(\Dcap\right)}\left(
     \lambda^{-1}\cal{M}^{\bullet},
     \R\sigma_{*}\I\left(\OB_{\dR}^{\dag,+}\right)
     \right)
   &\cong\R\sigma_{*}\R\shHom_{\sigma^{-1}\lambda^{-1}\I\left(\Dcap\right)}\left(
     \sigma^{-1}\lambda^{-1}\cal{M}^{\bullet},
     \I\left(\OB_{\dR}^{\dag,+}\right)
     \right) \\
   &=\R\sigma_{*}\R\shHom_{\nu^{-1}\I\left(\Dcap\right)}\left(
     \nu^{-1}\cal{M}^{\bullet},
     \I\left(\OB_{\dR}^{\dag,+}\right)
     \right),
\end{split}
\end{equation}
Thanks to~(\ref{eq:solexplicitdescription-MIC-constructmap}), we can now
produce the following adjunct of~(\ref{eq:solexplicitdescription-MIC-constructmap-anotherlemma})
\begin{equation*}
   \rL\sigma^{*}
   \R\shHom_{\lambda^{-1}\I\left(\Dcap\right)}\left(
     \lambda^{-1}\cal{M}^{\bullet},
     \I\left(\OB_{\dR}^{*,+}\right)
     \right)
   \to
   \R\shHom_{\nu^{-1}\I\left(\Dcap\right)}\left(
     \nu^{-1}\cal{M}^{\bullet},
     \I\left(\OB_{\dR}^{\dag,+}\right)
     \right)
\end{equation*}
with respect to the adjunction $\rL\sigma^{*}\dashv \R\sigma_{*}$,
cf.~\cite[\S 3.7]{Bo21}.
This is~(\ref{eq:solexplicitdescription-MIC})
in Lemma~\ref{lem:solexplicitdescription-MIC}:

\begin{lem}\label{lem:solexplicitdescription-MIC}
  Given a $\Dcap$-ind-Banach module $\cal{M}$ which is locally finite free as an $\cal{O}$-module,
  \begin{equation}\label{eq:solexplicitdescription-MIC}
   \Sol\left(\I\left(\cal{M}\right)\right)
   \isomap
     \R\shHom_{\nu^{-1}\I\left(\Dcap\right)}
      \left(
        \nu^{-1}\I\left(\cal{M}\right),
        \I\left(\OB_{\dR}^{\dag,+}\right)
      \right)
  \end{equation}
  is an isomorphism of $\I\left(\BB_{\dR}^{\dag,+}\right)$-module objects.
\end{lem}

\begin{proof}
  Firstly, we note that the canonical morphism
  $\sigma^{-1}\I\left(\BB_{\dR}^{*,+}\right)\to\I\left(\BB_{\dR}^{\dag,+}\right)$ is an isomorphism;
  this follows directly from the definitions and the cocontinuity of $\I$,
  cf.~\cite[Proposition 2.1.16]{Bo21}, which applies thanks to~\cite[Lemma 4.20]{Bo21}.
  Compute
  \begin{equation}\label{eq:computepullback-lem:solexplicitdescription-MIC}
  \begin{split}
    \rL\sigma^{*}\I\left(\OB_{\dR}^{*,+}\right)
    &=\I\left(\BB_{\dR}^{\dag,+}\right)
      \widehat{\otimes}_{\sigma^{-1}\I\left(\BB_{\dR}^{*,+}\right)}^{\rL}\sigma^{-1}\I\left(\OB_{\dR}^{*,+}\right) \\
    &\cong\sigma^{-1}\I\left(\OB_{\dR}^{*,+}\right)
    \cong\I\left(\sigma^{-1}\OB_{\dR}^{*,+}\right)
    \cong\I\left(\OB_{\dR}^{\dag,+}\right).
  \end{split}
  \end{equation}
  Next, we observe that $\cal{M}^{\bullet}$ admits a locally free resolution by
  $\Dcap$-ind-Banach modules, cf.~\cite[Theorem 6.1]{Bo21}; this is the Spencer resolution.
  By~\cite[Corollary 1.2.28]{Sch99}, we may apply $\I$ to get a locally
  free resolution of $\I\left(\cal{M}^{\bullet}\right)$ by $\I\left(\Dcap\right)$-module
  objects. And since $\lambda^{-1}$ and $\nu^{-1}$ are strongly exact,
  it suffices to check that~(\ref{eq:solexplicitdescription-MIC}) is an isomorphism
  for $\cal{M}=\Dcap$. This comes from
  \begin{equation*}
    \Sol\left(\I\left(\Dcap\right)\right)
    =\rL\sigma^{*}\I\left(\OB_{\dR}^{*,+}\right)
    \cong\I\left(\OB_{\dR}^{\dag,+}\right)
    \cong\R\shHom_{\nu^{-1}\I\left(\Dcap\right)}
      \left(
        \nu^{-1}\I\left(\Dcap\right),
        \I\left(\OB_{\dR}^{\dag,+}\right)
      \right),
  \end{equation*}
  where we have used~(\ref{eq:computepullback-lem:solexplicitdescription-MIC}).
\end{proof}

\begin{defn}
  Let $\cal{N}$ denote a vector bundle with flat connection $\nabla_{\cal{N}}$.
  This induces a connection
  \begin{equation*}
    \nabla\colon
    \nu^{-1}\mathcal{N}
    \widehat{\otimes}_{\nu^{-1}\mathcal{O}}
    \OB_{\dR}^{\dag,+}
    \to
    \nu^{-1}\mathcal{N}
    \widehat{\otimes}_{\nu^{-1}\mathcal{O}}
    \OB_{\dR}^{\dag,+}
    \widehat{\otimes}_{\nu^{-1}\mathcal{O}}\nu^{1}\Omega^{1}
  \end{equation*}
  defined by the formula
  $\nabla:=\left(\nu^{-1}\nabla_{\cal{N}}\right)\widehat{\otimes}\id + \id \widehat{\otimes}\nabla_{\dR}^{\dag,+}$.
  We write
  \begin{equation*}
    \left(\nu^{-1}\mathcal{N}\widehat{\otimes}_{\nu^{-1}\mathcal{O}}
    \OB_{\dR}^{\dag,+}\right)^{\nabla=0} := \ker\nabla.
  \end{equation*}
\end{defn}

\begin{prop}\label{prop:Solplus-and-Sol-for-IC}
  Let $\mathcal{M}$ denote a $\Dcap$-ind-Banach module
  which is locally finite free as an $\mathcal{O}$-module.
  $\Sol\left( \mathcal{M} \right)$ is concentrated in degree $0$ and
  \begin{equation*}
    \Ho^{0}\left(\Sol\left( \I\left(\mathcal{M}\right) \right)\right)
    \cong
    \I\left(\left(\nu^{-1}\DD\left(\mathcal{M}\right)
    \widehat{\otimes}_{\nu^{-1}\mathcal{O}}
    \OB_{\dR}^{\dag,+}\right)^{\nabla=0}\right).
  \end{equation*}
\end{prop}

\begin{proof}[Proof of Proposition~\ref{prop:Solplus-and-Sol-for-IC}]
  To increase readability, we omit the symbol $\I$.
  All operations are carried out in the left heart, or the derived category of the left heart.
  Implicitly, we often invoke that fact that $\I$ sends
  strictly exact complexes to exact complexes, cf.~\cite[Corollary 1.2.28]{Sch99}.
  
  Recall~\cite[Proposition 7.3]{Bo21}, which gives a functorial isomorphism
  \begin{equation}\label{eq:duality-IC}
    \DD\left(\mathcal{M}\right)\cong\shHom_{\mathcal{O}}\left(\mathcal{M},\mathcal{O}\right).
  \end{equation}
  Next, $S^{\bullet}\to\mathcal{O}$ denotes the Spencer resolution, cf.~\cite[Theorem 6.1]{Bo21}.
  We have $S^{-i}=\Dcap\widehat{\otimes}_{\mathcal{O}}\wedge^{i}\mathcal{T}$ for all $i\in\NN$,
  and $S^{-i}=0$ for all $i<0$. Here, $\mathcal{T}$ denotes the tangent sheaf,
  cf. Definition~\ref{defn:tangentsheaf}. The Spencer resolution
  is a locally free resolution of $\Dcap$-ind-Banach modules.
  It is thus a resolution by strongly flat $\mathcal{O}$-ind-Banach modules,
  Since $\nu^{-1}$ is strongly exact,
  cf.~\cite[\S 3.7]{Bo21},
  $\nu^{-1}S^{\bullet}\to\nu^{-1}\mathcal{O}$ is a locally free solution of
  $\nu^{-1}\Dcap$-ind-Banach modules. It follows from
  \cite[Corollary 5.36 and Lemma 6.5]{Bo21}
  that it is a resolution by strongly flat $\nu^{-1}\mathcal{O}$-modules.
  Now compute:
  \begin{equation}\label{eq:dRplus-for-IC-computation-1}
  \begin{split}
    &\Sol\left(\mathcal{M}\right) \\
    &=\R\shHom_{\nu^{-1}\Dcap}\left(\nu^{-1}\mathcal{M},\OB_{\dR}^{\dag,+}\right) \\
    &=\R\shHom_{\nu^{-1}\Dcap}\left(
     \nu^{-1}\mathcal{M}\widehat{\otimes}_{\nu^{-1}\mathcal{O}}^{\rL}\nu^{-1}\mathcal{O},
     \OB_{\dR}^{\dag,+}\right) \\
    &\cong\R\shHom_{\nu^{-1}\Dcap}\left(
     \nu^{-1}\mathcal{M}\widehat{\otimes}_{\nu^{-1}\mathcal{O}}\nu^{-1}S^{\bullet},
     \OB_{\dR}^{\dag,+}\right) \\
    &\cong
    \R\shHom_{\nu^{-1}\Dcap}\left(
     \nu^{-1}\mathcal{M}\widehat{\otimes}_{\nu^{-1}\mathcal{O}}
     \left(\nu^{-1}\Dcap\widehat{\otimes}_{\nu^{-1}\mathcal{O}}\nu^{-1}\kern -3,00pt\wedge^{\bullet}\mathcal{T}\right),
     \OB_{\dR}^{\dag,+}\right) \\
    &\cong\R\shHom_{\nu^{-1}\mathcal{O}}\left(
     \nu^{-1}\mathcal{M}\widehat{\otimes}_{\nu^{-1}\mathcal{O}}
     \nu^{-1}\kern -3,00pt\wedge^{\bullet}\mathcal{T},
     \OB_{\dR}^{\dag,+}\right).
     \end{split}
  \end{equation}
  Since $\nu^{-1}\mathcal{M}\widehat{\otimes}_{\nu^{-1}\mathcal{O}}
     \nu^{-1}\kern -3,00pt\wedge^{\bullet}\mathcal{T}$
  is a complex of locally finite free $\nu^{-1}\mathcal{O}$-modules.
  Thus it computes the $\R\shHom$.
  Now continue with the computation:
  \begin{equation}\label{eq:dRplus-for-IC-computation-2}
  \begin{split}
    \Sol\left(\mathcal{M}\right)
    &\stackrel{\text{(\ref{eq:dRplus-for-IC-computation-1})}}{\cong}
    \shHom_{\nu^{-1}\mathcal{O}}\left(
     \nu^{-1}\mathcal{M}\widehat{\otimes}_{\nu^{-1}\mathcal{O}}
     \nu^{-1}\kern -3,00pt\wedge^{\bullet}\mathcal{T},
     \OB_{\dR}^{\dag,+}\right) \\
    &\stackrel{\text{(\ref{eq:duality-IC})}}{\cong}
    \nu^{-1}\DD\left(
     \mathcal{M}\right)
     \widehat{\otimes}_{\nu^{-1}\mathcal{O}}
     \left(
     \OB_{\dR}^{\dag,+}\widehat{\otimes}_{\nu^{-1}\mathcal{O}}\nu^{-1}\Omega^{\bullet}
     \right).
     \end{split}
  \end{equation}
  Because $\DD\left(\mathcal{M}\right)$
  is locally finite free as an $\mathcal{O}$-module by~(\ref{eq:duality-IC}),
  $\nu^{-1}\DD\left(\mathcal{M}\right)$ is strongly flat as
  a $\nu^{-1}\mathcal{O}$-ind-Banach module. It follows
  from~(\ref{eq:dRplus-for-IC-computation-2})
  and Theorem~\ref{thm:pioncare-lem}
  that $\Sol\left(\mathcal{M}\right)$
  is concentrated in degree zero.
  More precisely,
  \begin{equation*}
  \begin{split}
    \Ho^{0}\left(\Sol\left(\mathcal{M}\right)\right)
    &\stackrel{\text{(\ref{eq:dRplus-for-IC-computation-2})}}{\cong}
     \Ho^{0}\left(\nu^{-1}\DD\left(
     \mathcal{M}\right)
     \widehat{\otimes}_{\nu^{-1}\mathcal{O}}
     \left(
     \OB_{\dR}^{\dag,+}\widehat{\otimes}_{\nu^{-1}\mathcal{O}}\nu^{-1}\Omega^{\bullet}
     \right)\right) \\
     &=\left(\nu^{-1}\DD\left(\mathcal{M}\right)
    \widehat{\otimes}_{\nu^{-1}\mathcal{O}}
    \OB_{\dR}^{\dag,+}\right)^{\nabla=0}.\qedhere
   \end{split}
  \end{equation*}
\end{proof}

Recall that $d$ is the dimension of $X$.
We consider $\dRfunctor=\dRfunctorB_{\dR}^{\dag,+}$.

\begin{cor}\label{cor:dRplus-and-dR-for-IC}
  Let $\mathcal{M}$ denote a $\Dcap$-ind-Banach module
  which is locally finite free as an $\mathcal{O}$-module.
  Then $\dRfunctor\left(\mathcal{M}\right)$ is concentrated in degree $-d$ and
  \begin{equation*}
    \Ho^{-d}\left(\dRfunctor\left( \I\left(\mathcal{M}\right) \right)\right)
    \cong\I\left(\left(\nu^{-1}\mathcal{M}
    \widehat{\otimes}_{\nu^{-1}\mathcal{O}}
    \OB_{\dR}^{\dag,+}\right)^{\nabla=0}\right).
  \end{equation*}
\end{cor}

\begin{proof}
  Again, we omit $\I$. Using Proposition~\ref{prop:Solplus-and-Sol-for-IC}
  and $\DD^{2}\left(\mathcal{M}\right)\cong\mathcal{M}$,
  (see~\cite[Theorem 9.17]{Bo21}), we find
  \begin{equation*}
  \begin{split}
    \Ho^{-d}\left( \dRfunctor\left(\mathcal{M}\right)\right)
    =\Ho^{0}\left(\Sol\left( \DD\left(\mathcal{M}\right) \right)\right)
    \cong
    \left(\nu^{-1}\DD^{2}\left(\mathcal{M}\right)\widehat{\otimes}_{\nu^{-1}\mathcal{O}}\OB_{\dR}^{\dag,+}\right)^{\nabla=0} \\
    \cong
    \left(\nu^{-1}\mathcal{M}\widehat{\otimes}_{\nu^{-1}\mathcal{O}}\OB_{\dR}^{\dag,+}\right)^{\nabla=0}.\qedhere
  \end{split}
  \end{equation*}
\end{proof}

\subsection{Compatibility with Scholze's horizontal sections functor}
\label{subsec:Comp-with-Scholzesfunctor}

Now we can relate Scholze's horizontal sections functor to our de Rham functor.
Any vector bundle with flat connection $\cal{E}$ on $X$ is canonically a coadmissible
$\Dcap$-module by~\cite[Theorem B]{MR3846550}. Next,
Theorem~\cite[Theorem 4.4]{Bo21} allows us to view $\cal{E}$ as a $\Dcap$-ind-Banach module.
We can therefore consider the object $\Ho^{-d}\left(\dRfunctor\left( \I\left(\mathcal{E}\right) \right)\right)$.
This is a sheaf on $X_{\proet}$ valued in the left heart $\IndBan_{\I\left(k_{0}\right)}$.
Now compute
\begin{equation*}
  \begin{split}
   |\Ho^{-d}\left(\dRfunctor\left( \I\left(\mathcal{E}\right) \right)\right)|
   &\stackrel{\text{\ref{cor:dRplus-and-dR-for-IC}}}{\cong}
     |\I\left(\left( \nu^{-1}\mathcal{E} \widehat{\otimes}_{\nu^{-1}\mathcal{O}}\OB_{\dR}^{\dag,+}\right)^{\nabla=0}\right)|
   \stackrel{\ref{lem:underlyingabstractpresheaf-is-sheaf}}{\cong}
   |\left( \nu^{-1}\mathcal{E} \widehat{\otimes}_{\nu^{-1}\mathcal{O}}\OB_{\dR}^{\dag,+}\right)^{\nabla=0}| \\
   &\cong\left( \nu^{-1}\mathcal{E} \otimes_{\nu^{-1}\mathcal{O}}|\OB_{\dR}^{\dag,+}|\right)^{\nabla=0}
  \end{split}
\end{equation*}
We note that the tensor product in the second line does not have to be completed
because $\nu^{-1}\mathcal{E}$ is a locally finite free sheaf of $\nu^{-1}\mathcal{O}$-modules.
Together with Theorem~\ref{thm:Cauchy}, this implies:

\begin{thm}[Compatibility with Scholze's horizontal sections functor]\label{thm:compatibility-scholze-cauchy}
There is a canonical isomorphism of $\BB_{\dR}^{+}$-local systems
\begin{equation*}
|\Ho^{-d}\left(\dRfunctor\left( \I\left(\mathcal{E}\right) \right)\right)|
\isomap
\left(
\nu^{-1}\cal E
\otimes_{\nu^{-1}\cal O}
\OB_{\dR}^{+}
\right)^{\nabla=0}
\end{equation*}
\end{thm} 

\appendix


\section{Forgetting ind-Banach structures}
\label{sec:forgetIB}

Let us fix the notation as in~\cite[\S\ref*{ch:functional-analysis}]{WiersigPeriods}.

Let $R$ be a Banach ring. Then the obtain a category of $R$-Banach modules
$\Ban_{R}$. It admits a forgetful functor $M\mapsto|M|$ which sends $M$ to its underlying
abstract ring $|R|$. It extends to a functor
\begin{equation}\label{eq:underlyingabstractmodule}
  \IndBan_{R} \to \Mod(|R|), \text{``}\varinjlim_{i}\text{"}M_{i}\mapsto\varinjlim_{i}|M_{i}|
\end{equation}

\begin{lem}\label{lem:underlyingabstractmodule-commuteswithfinitelimits}
  The functor~(\ref{eq:underlyingabstractmodule})
  commutes with finite limits.
\end{lem}

\begin{proof}
  Filtered colimits commute with finite limits by~\cite[Corollary~\ref*{cor:filteredcol-inIndBan-stronglyexact}]{WiersigPeriods},
  thus it suffices to show that the functor $\Ban_{R}\to\Mod(R)$, $M\mapsto|M|$ commutes
  with finite limits. In fact, it suffices to compute that it commutes with finite products and equalisers,
  which is clear.
\end{proof}

\begin{defn}\label{defn:underlyingabstractpresheaf}
  Given a sheaf $\cal{F}$ of $R$-ind-Banach algebras on $X$,
  $|\mathcal{F}|$ is the sheafification of the presheaf $U\mapsto|\mathcal{F}(U)|$
  of abstract $|R|$-modules.
\end{defn}

\begin{lem}\label{lem:underlyingabstractpresheaf-is-sheaf}
  Given a sheaf $\cal{F}$ of $R$-ind-Banach algebras on $X$,
  the canonical map
  $|\mathcal{F}(U)|\isomap|\mathcal{F}|(U)$ is an isomorphism of abstract
  $|R|$-modules if $X$ has a basis by quasi-compact objects.
\end{lem}

\begin{proof}
  This follows from Lemma~\ref{lem:underlyingabstractmodule-commuteswithfinitelimits}.
\end{proof}

We write $\IndBan_{\I\left(R\right)}:=\LH\left(\IndBan_{R}\right)$ for the left heart.

\begin{lem}\label{lem:LH-underlyingabstractmodule}
  Up to equivalence, there exists a unique right exact functor
  \begin{equation*}
    \IndBan_{\I\left(R\right)}\to\Mod(|R|), E\mapsto|E|
  \end{equation*}
  such that $|M|\cong|\I\left(M\right)|$ for all $R$-ind-Banach modules
  $M$, functorially.
\end{lem}

\begin{proof}
  Thanks to Lemma~\ref{lem:underlyingabstractmodule-commuteswithfinitelimits},
  we may apply~\cite[Proposition 1.2.34]{Sch99}.
\end{proof}

\begin{defn}\label{defn:underlyingabstractsheaf-LH}
  Given a sheaf $\cal{G}$ of $\I\left(R\right)$-ind-Banach algebras on
  $X$, $|\cal{G}|$ is the sheafification of the presheaf $U\mapsto|\mathcal{G}(U)|$
  of abstract $|R|$-modules.
\end{defn}

\begin{lem}\label{lem:underlyingabstractpresheaf-is-sheaf}
  Given a sheaf $\cal{F}$ of $R$-ind-Banach algebras on $X$,
  there is a canonical isomorphism
  $|\I\left(\mathcal{F}\right)|\isomap|\mathcal{F}|$
  of sheaves of $|R|$-modules if $X$ has only finite coverings.
\end{lem}

\begin{proof}
  Let $U\in X$ be arbitrary and compute
  \begin{equation*}
    |\I\left(\mathcal{F}\right)|(U)
    \cong|\I\left(\mathcal{F}\right)(U)|
    =|\I\left(\mathcal{F}(U)\right)|
    \cong|\mathcal{F}(U)|
    \stackrel{\text{\ref{lem:underlyingabstractpresheaf-is-sheaf}}}{\cong}|\mathcal{F}|(U).
  \end{equation*}
  The first isomorphism comes from the right exactness of $|\cdot|$,
  see Lemma~\ref{lem:LH-underlyingabstractmodule}.
\end{proof}

\bibliography{RecThm}
\bibliographystyle{amsplain}  

\end{document}